\documentclass[11pt]{amsart}
\usepackage{amsfonts,amssymb,amscd,amsmath,enumerate,verbatim,calc,latexsym,pstcol,pst-plot,pst-3d,}
\def\NZQ{\mathbb}               

\def\ZZ{{\NZQ Z}}
\def\RR{{\NZQ R}}

\def\frk{\mathfrak}               

\def\Phi{{\frk N}}
\def\ab{{\bold a}}
\def\eb{{\bold e}}

\def\tb{{\bold t}}
\def\xb{{\bold x}}

\def\a{\alpha}

\def\opn#1#2{\def#1{\operatorname{#2}}} 
\opn\chara{char} 
\opn\length{\ell} 
\opn\pd{pd} 
\opn\rk{rk}
\opn\projdim{proj\,dim} 
\opn\injdim{inj\,dim} 
\opn\rank{rank}
\opn\depth{depth} 
\opn\grade{grade} 
\opn\height{height}
\opn\embdim{emb\,dim} 
\opn\codim{codim}

\opn\Tr{Tr} 
\opn\bigrank{big\,rank}
\opn\superheight{superheight}
\opn\lcm{lcm}
\opn\trdeg{tr\,deg}
\opn\reg{reg} 
\opn\lreg{lreg} 
\opn\ini{in} 
\opn\lpd{lpd}
\opn\size{size}
\opn\mult{mult}
\opn\dist{dist}
\opn\cone{cone}
\opn\lex{lex}
\opn\rev{rev}
\DeclareMathOperator{\K-d}{Krull-dim}
\opn\div{div} \opn\Div{Div} \opn\cl{cl} \opn\Cl{Cl}
\opn\Spec{Spec} \opn\Supp{Supp} \opn\supp{supp} \opn\Sing{Sing}
\opn\Ass{Ass} \opn\Min{Min}
\opn\Ann{Ann} \opn\Rad{Rad} \opn\Soc{Soc}
\opn\Syz{Syz} \opn\Im{Im} \opn\Ker{Ker} \opn\Coker{Coker}
\opn\Am{Am} \opn\Hom{Hom} \opn\Tor{Tor} \opn\Ext{Ext}
\opn\End{End} \opn\Aut{Aut} \opn\id{id} \opn\ini{in}

\opn\nat{nat}
\opn\pff{pf}
\opn\Pf{Pf} \opn\GL{GL} \opn\SL{SL} \opn\mod{mod} \opn\ord{ord}
\opn\Gin{Gin}
\opn\Hilb{Hilb}\opn\adeg{adeg}\opn\std{std}\opn\ip{infpt}
\opn\Pol{Pol}
\opn\sat{sat}
\opn\Var{Var}
\opn\Gen{Gen}

\opn\aff{aff} \opn\con{conv} \opn\relint{relint} \opn\st{st}
\opn\lk{lk} \opn\cn{cn} \opn\core{core} \opn\vol{vol}
\opn\link{link} \opn\star{star}
\opn\gr{gr}

\def\pot#1#2{#1[\kern-0.28ex[#2]\kern-0.28ex]}

\opn\dirlim{\underrightarrow{\lim}}
\opn\inivlim{\underleftarrow{\lim}}
\let\union=\cup

\let\to=\rightarrow

\def\Implies{\ifmmode\Longrightarrow \else
        \unskip${}\Longrightarrow{}$\ignorespaces\fi}
\def\implies{\ifmmode\Rightarrow \else
        \unskip${}\Rightarrow{}$\ignorespaces\fi}
\def\iff{\ifmmode\Longleftrightarrow \else
        \unskip${}\Longleftrightarrow{}$\ignorespaces\fi}

\let\:=\colon
\newtheorem{Theorem}{Theorem}[section]
\newtheorem{Lemma}[Theorem]{Lemma}
\newtheorem{Corollary}[Theorem]{Corollary}

\newtheorem{Conjecture}[Theorem]{Conjecture}

\let\epsilon\varepsilon
\let\phi=\varphi
\let\kappa=\varkappa
\def\qed{\ifhmode\textqed\fi
      \ifmmode\ifinner\quad\qedsymbol\else\dispqed\fi\fi}
\def\textqed{\unskip\nobreak\penalty50
       \hskip2em\hbox{}\nobreak\hfil\qedsymbol
       \parfillskip=0pt \finalhyphendemerits=0}
\def\dispqed{\rlap{\qquad\qedsymbol}}

\opn\dis{dis}
\opn\height{height}
\opn\dist{dist}
\def\pnt{{\raise0.5mm\hbox{\large\bf.}}}

\opn\Lex{Lex}

\begin{document}
\title[]
{Depth of initial ideals of normal edge rings}
\author[]{Takayuki Hibi}
\address[]{Department of Pure and Applied Mathematics,
Graduate School of Information Science and Technology,
Osaka University,
Toyonaka, Osaka 560-0043, Japan}
\email{hibi@math.sci.osaka-u.ac.jp}
\author[]{Akihiro Higashitani}
\address[]{Department of Pure and Applied Mathematics,
Graduate School of Information Science and Technology,
Osaka University,
Toyonaka, Osaka 560-0043, Japan}
\email{sm5037ha@ecs.cmc.osaka-u.ac.jp}
\author[]{Kyouko Kimura}
\address[]{Department of Mathematics, Faculty of Science, 
Shizuoka University,
Shizuoka 422-8529, Japan}
\email{skkimur@ipc.shizuoka.ac.jp}
\author[]{Augustine B. O'Keefe}
\address[]{Mathematics Department,
Tulane University,
6823 St.\  Charles Ave,
New Orleans, LA 70118, U.S.A}
\email{aokeefe@tulane.edu}
\thanks{
\hspace{0.43cm}
}
\begin{abstract}
Let $G$ be a finite graph on the vertex set 
$[d] = \{ 1, \ldots, d \}$
with the edges $e_1, \ldots, e_n$
and $K[\tb] = K[t_1, \ldots, t_d]$ the polynomial ring
in $d$ variables over a field $K$.
The edge ring of $G$ is the semigroup ring $K[G]$ 
which is generated by those monomials $\tb^e = t_it_j$
such that $e = \{i, j\}$ is an edge of $G$.
Let $K[\xb] = K[x_1, \ldots, x_n]$ be the polynomial ring 
in $n$ variables over $K$ and define the surjective homomorphism 
$\pi : K[\xb] \to K[G]$ by setting
$\pi(x_i) = \tb^{e_i}$ for $i = 1, \ldots, n$. 
The toric ideal $I_G$ 
of $G$ is the kernel of $\pi$. It will be proved that, 
given integers $f$ and $d$ with $6 \leq f \leq d$,
there exist a finite connected nonbipartite graph $G$ 
on $[d]$ together with a reverse lexicographic order $<_{\rev}$
on $K[\xb]$ and a lexicographic order $<_{\lex}$ 
on $K[\xb]$
such that (i) $K[G]$ is normal, (ii) 
$\depth K[\xb]/\ini_{<_{\rev}}(I_G) = f$
and (iii) $K[\xb]/\ini_{<_{\lex}}(I_G)$ is 
Cohen--Macaulay,
where $\ini_{<_{\rev}}(I_G)$ (resp.\  $\ini_{<_{\lex}}(I_G)$) is the initial ideal of $I_G$ 
with respect to $<_{\rev}$ (resp.\  $<_{\lex}$) and 
where $\depth K[\xb]/\ini_{<_{\rev}}(I_G)$ is the depth of
$K[\xb]/\ini_{<_{\rev}}(I_G)$.
\end{abstract}
\maketitle

\section*{Introduction}
The study on edge rings \cite{OH98} of finite graphs
together with their toric ideals \cite{OH99} 
have been achieved from viewpoints of both commutative algebra
and combinatorics.  Following the previous paper \cite{HHKO10},
which investigated a question about depth of
edge rings, 
the topic of the present paper is depth of 
initial ideals of normal edge rings.

Let $G$ be a finite simple graph,
i.e., a finite graph with no loop and no multiple edge, 
on the vertex set
$[d] = \{ 1, \ldots, d \}$ and
$E(G) = \{ e_1, \ldots, e_n \}$ its edge set.
Let $K[\tb] = K[t_1, \ldots, t_d]$ be the polynomial ring
in $d$ variables over a field $K$
and write $K[G]$ for the subring of $K[\tb]$ generated by those
monomials $\tb^e = t_it_j$
with $e = \{ i, j \} \in E(G)$.
The semigroup ring $K[G]$ is called 
the {\em edge ring} of $G$.  
Let $K[\xb] = K[x_1, \ldots, x_n]$ be the polynomial ring
in $n$ variables over $K$.
The {\em toric ideal} of $G$ is
the kernel $I_G$ of the surjective ring homomorphism 
$\pi : K[\xb] \to K[G]$ defined by setting
$\pi(x_i) = \tb^{e_i}$ for $i = 1, \ldots, n$. 
Thus in particular one has $K[G] \cong K[\xb]/I_G$.
If $G$ is connected and nonbipartite
(resp.\  connected and bipartite), then $\K-d K[G] = d$ 
(resp.\  $\K-d K[G] = d - 1$), where 
$\K-d K[G]$ stands for the Krull dimension of $K[G]$.

It follows from the criterion \cite[Corollary 2.3]{OH98}
that the edge ring $K[G]$ of a connected graph 
$G$ is normal if and only if, for
any two odd cycles $C_1$ and $C_2$ of $G$ having no common vertex, 
there exists an edge $\{v, w\}$ of $G$ 
such that $v$ is a vertex of $C_1$ and $w$ is a vertex of $C_2$.

We refer the reader to \cite[Chapter 2]{HerzogHibi}
for fundamental materials on Gr\"obner bases.
Let $<$ be a monomial order on $K[\xb]$ and
$\ini_<(I_G)$ the initial ideal of $I_G$ 
with respect to $<$.
The topic of this paper is  
$\depth K[\xb]/\ini_<(I_G)$,
the depth of $K[\xb]/\ini_<(I_G)$,
when $K[G]$ is normal.  Computational experience yields the following

\begin{Conjecture}
\label{depthconjecture}
{\em
Let $G$ be a finite connected nonbipartite graph on $[d]$
with $d \geq 6$ and
suppose that its edge ring $K[G]$ is normal.  Then 
$\depth K[\xb]/\ini_<(I_G) \geq 6$ for any monomial order $<$
on $K[\xb]$.
}
\end{Conjecture}

Now, even though Conjecture \ref{depthconjecture} is completely open, 
by taking Conjecture \ref{depthconjecture} into consideration,
this paper will be devoted to proving the following 

\begin{Theorem}
\label{main}
Given integers $f$ and $d$ with $6 \leq f \leq d$,
there exists a finite connected nonbipartite graph $G$ 
on $[d]$ together with a reverse lexicographic order $<_{\rev}$
on $K[\xb]$ and a lexicographic order $<_{\lex}$ 
on $K[\xb]$
such that 
\begin{enumerate}
\item[{\rm (i)}]$K[G]$ is normal;
\item[{\rm (ii)}] 
$\depth K[\xb]/\ini_{<_{\rev}}(I_G) = f$;
\item[{\rm (iii)}]
$K[\xb]/\ini_{<_{\lex}}(I_G)$ is Cohen--Macaulay.
\end{enumerate}
\end{Theorem}  

Let $k \geq 1$ be an arbitrary integer.  We introduce 
the finite connected nonbipartite graph $G_{k+5}$ on $[k+5]$ 
which is drawn in Figure $0.1$.
Clearly, the edge ring $K[G_{k+5}]$ is normal.
It will turn out that 
$G_{k+5}$ plays an important role
in our proof of Theorem \ref{main}. 

\medskip

\begin{center}
  \begin{picture}(300,160)(0,-50)
    \put(150,90){\circle*{5}}
    \put(150,70){\circle*{5}}
    \put(150,45){\circle*{1}}
    \put(150,40){\circle*{1}}
    \put(150,35){\circle*{1}}
    \put(150,10){\circle*{5}}
    \put(70,50){\circle*{5}} 
    \put(70,-30){\circle*{5}} 
    \put(20,10){\circle*{5}} 
    \put(230,50){\circle*{5}} 
    \put(230,-30){\circle*{5}} 
    \put(280,10){\circle*{5}} 
    \put(150,95){$5$}
    \put(150,75){$6$}
    \put(150,-3){$k+3$}
    \put(60,55){$1$}
    \put(60,-40){$3$}
    \put(-10,10){$k+4$}
    \put(233,55){$2$}
    \put(233,-40){$4$}
    \put(285,10){$k+5$}
    \thinlines
    \put(70,50){\line(2,1){80}}
    \put(70,50){\line(4,1){80}}
    \put(70,50){\line(2,-1){80}}
    \put(70,-30){\line(-5,4){50}}
    \put(70,50){\line(-5,-4){50}}
    \put(230,50){\line(-2,1){80}}
    \put(230,50){\line(-4,1){80}}
    \put(230,50){\line(-2,-1){80}}
    \put(230,-30){\line(5,4){50}}
    \put(230,50){\line(5,-4){50}}
    \put(70,50){\line(0,-1){80}}
    \put(230,50){\line(0,-1){80}}
    \put(70,-30){\line(1,0){160}}
    \put(110,80){$e_2$}
    \put(110,52.5){$e_3$}
    \put(110,20){$e_k$}
    \put(22,33){$e_{k+1}$}
    \put(22,-17){$e_{2k+4}$}
    \put(180,80){$e_{k+3}$}
    \put(180,52.5){$e_{k+4}$}
    \put(180,20){$e_{2k+1}$}
    \put(255,33){$e_{2k+2}$}
    \put(255,-17){$e_{2k+5}$}
    \put(75,0){$e_1$}
    \put(207,0){$e_{k+2}$}
    \put(140,-40){$e_{2k+3}$}
  \end{picture}
\newline
{\bf Figure 0.1.} (finite graph $G_{k+5}$)
\end{center}

\bigskip

The essential step in order to prove Theorem \ref{main} is
to show the following

\begin{Lemma}
\label{reverselex}
Let $<_{\rev}$ (resp.\  $<_{\lex}$)
denote the reverse lexicographic order
(resp.\  the lexicographic order)
on $K[\xb] = K[x_1, \ldots, x_{2k+5}]$
induced by the ordering $x_1 > \cdots > x_{2k+5}$
of the variables.  Then 
\begin{enumerate}
\item[{\rm (i)}]
$\depth K[\xb]/\ini_{<_{\rev}}(I_{G_{k+5}}) = 6$;
\item[{\rm (ii)}]
$K[\xb]/\ini_{<_{\lex}}(I_{G_{k+5}})$ is Cohen--Macaulay.
\end{enumerate}
\end{Lemma}

Once we establish Lemma \ref{reverselex}, to prove Theorem \ref{main} 
is straightforward. 
In fact, given integers $f$ and $d$ with $6 \leq f \leq d$,
we define the finite graph $\Gamma$ on $[d - f + 6]$
to be $G_{d - f + 6}$
with the edges $e_1, e_2, \ldots, e_{2(d-f)+7}$
and then introduce the finite connected nonbipartite graph $G$ on $[d]$ 
which is obtained from $\Gamma$ by adding $f - 6$ edges
\[
e_{2(d-f)+7 + i} = \{ 1, d - f + 6 + i \},
\, \, \, \, \, i = 1, \ldots, f - 6 
\]
to $\Gamma$.
Clearly, both edge rings $K[\Gamma]$ and $K[G]$ are normal, and
\[
I_G = I_\Gamma (K[x_1, \ldots, x_{2d-f+1}]).
\]
Thus in particular
\[
\ini_{<}(I_G) = 
\ini_{<}(I_\Gamma)(K[x_1, \ldots, x_{2d-f+1}]),
\]
where $<$ is any monomial order on 
$K[x_1, \ldots, x_{2d-f+1}]$.
Thus Lemma \ref{reverselex} guarantees that 
\[
\depth K[x_1, \ldots, x_{2d-f+1}]/\ini_{<_{\rev}}(I_G) = f 
\]
and $K[x_1, \ldots, x_{2d-f+1}]/\ini_{<_{\lex}}(I_G)$ is 
Cohen--Macaulay, as desired.

\section{Preliminaries}

Let $G=G_{k+5}$. 
In this section, we will find a Gr\"{o}bner basis of $I_G$ 
and a set of generators of the initial ideal of $I_G$ 
with respect to the reverse lexicographic order. 

Let $K[\xb]=K[x_1,\ldots,x_{2k+5}]$ be the polynomial ring 
in $2k+5$ variables over a field $K$. 
There are 4 kinds of primitive even closed walks of $G$:
\begin{enumerate}
\item a 4-cycle: $(e_i,e_{k+1+i},e_{k+1+j},e_j)$, where $2 \leq i < j \leq k$; 
\item a walk on two 3-cycles and the same edge $e_{2k+3}$ combining two cycles: \\
$(e_1,e_{k+1},e_{2k+4},e_{2k+3},e_{k+2},e_{2k+2},e_{2k+5},e_{2k+3})$; 
\item a 6-cycle: $(e_i,e_{k+1+i}, e_{k+2}, e_{2k+3}, e_{2k+4}, e_{k+1})$ or 
$(e_i, e_{k+1+i}, e_{2k+2}, e_{2k+5}, e_{2k+3}, e_1)$, where $2 \leq i \leq k$; 
\item a walk on two 3-cycles and the length $2$ paths combining two cycles: 
\newline
$(e_{k+2},e_{2k+5},e_{2k+2},e_{k+1+i},e_{i},e_{1},e_{2k+4},e_{k+1},e_{j},e_{k+1+j})$, 
where $2 \leq i \leq j \leq k$. 
\end{enumerate}

It was proved in \cite[Lemma 3.1]{OH99} that 
the binomials corresponding to these primitive even closed walks 
generate the toric ideal $I_G$. 
Let $<_{\rev}$ be the reverse lexicographic order with $x_1>x_2> \cdots > x_{2k+5}$. 

\begin{Lemma}\label{Grobner}
The set of binomials corresponding to primitive even closed walks 
(1), (2), (3) and (4) is a Gr\"{o}bner basis of $I_G$ 
with respect to $<_{\rev}$. 
\end{Lemma}
\begin{proof}
The result follows from a direct application of Buchberger's criterion 
to the set of generators of $I_G$ corresponding to the primitive even closed walks 
listed above. Let $f$ and $g$ be two such generators. 
We will prove that the $S$-polynomial $S(f,g)$ will reduce to 0 
by generators of type (1), (2), (3) and (4). 
Let $i,j,p,q$ be integers with $2 \leq i,j,p,q \leq k$. 
On the following proof, we will underline the leading monomial of a binomial 
with respect to $<_{\rev}$. 
\newline
{\bf Case 1:} Let $f=x_ix_{k+1+j}-\underline{x_jx_{k+1+i}}$ 
and $g=x_px_{k+1+q}-\underline{x_qx_{k+1+p}}$ 
be generators of type (1), where $i < j$ and $p < q$. 
If $i \not= p$ and $j \not= q$, then the leading monomials of $f$ and $g$ are coprime. 
Thus $S(f,g)$ will reduce to 0. We assume that $i=p$. Then 
\begin{eqnarray*}
S(f,g) &=& \frac{\lcm(LM_{<_{\rev}} (f), LM_{<_{\rev}} (g))}{LT_{<_{\rev}}(f)}f
-\frac{\lcm(LM_{<_{\rev}} (f), LM_{<_{\rev}} (g))}{LT_{<_{\rev}}(g)}g\\
&=& -x_q(x_ix_{k+1+j}-x_jx_{k+1+i})- (-x_j)(x_ix_{k+1+q}-x_qx_{k+1+i}) \\
&=& -x_i(x_qx_{k+1+j}-x_jx_{k+1+q}). 
\end{eqnarray*}
Note that, up to sign, $x_qx_{k+1+j}-x_jx_{k+1+q}$ is a generator of $I_G$ of type (1).
Therefore $S(f,g)$ will reduce to 0. The case of $j=q$ is similar. \\
{\bf Case 2:} 
Let $f$ be the same as above and 
$g=x_1x_{k+2}x_{2k+4}x_{2k+5}-\underline{x_{k+1}x_{2k+2}x_{2k+3}^2}$
a generator of type (2). Since $2 \leq i < j \leq k$, 
the leading monomials of $f$ and $g$ are always coprime. \\
{\bf Case 3:} Again, we set that $f$ is the same as above. Let $g$ be of type (3). 
First, let $g=x_px_{k+2}x_{2k+4}-\underline{x_{k+1}x_{k+1+p}x_{2k+3}}$. 
If $i \not= p$, then the leading monomials of $f$ and $g$ are coprime. 
We assume that $i=p$. Then 
\begin{eqnarray*}
S(f,g)&=& -x_{k+1}x_{2k+3}f - (-x_j)g \\
&=& -x_ix_{k+1}x_{k+1+j}x_{2k+3}+x_ix_jx_{k+2}x_{2k+4} \\
&=& x_i(x_jx_{k+2}x_{2k+4}-x_{k+1}x_{k+1+j}x_{2k+3}), 
\end{eqnarray*}
where $x_jx_{k+2}x_{2k+4}-x_{k+1}x_{k+1+j}x_{2k+3}$ is of type (3). 
Next, let $g=\underline{x_px_{2k+2}x_{2k+3}}-x_1x_{k+1+p}x_{2k+5}$. 
If $j \not= p$, then the leading monomials of $f$ and $g$ are coprime. 
We assume that $j=p$. Then 
\begin{eqnarray*}
S(f,g)&=& -x_{2k+2}x_{2k+3}f - x_{k+1+i}g \\
&=& -x_{k+1+j}(x_ix_{2k+2}x_{2k+3}-x_1x_{k+1+i}x_{2k+5}) 
\end{eqnarray*}
and again we have that $x_ix_{2k+2}x_{2k+3}-x_1x_{k+1+i}x_{2k+5}$ is of type (3). \\
{\bf Case 4:} Again, we assume that $f$ is the same as above. 
Let $g=\underline{x_px_qx_{k+2}x_{2k+2}x_{2k+4}}-x_1x_{k+1}x_{k+1+p}x_{k+1+q}x_{2k+5}$ 
be of type (4), where $p \leq q$. 
If $j \not= p$ and $j \not= q$, then the leading monomials of $f$ and $g$ are coprime. 
If $j=p$, then 
\begin{eqnarray*}
S(f,g)&=& -x_qx_{k+2}x_{2k+2}x_{2k+4}f - x_{k+1+i}g \\
&=& -x_{k+1+j}(x_ix_qx_{k+2}x_{2k+2}x_{2k+4}-x_1x_{k+1}x_{k+1+i}x_{k+1+q}x_{2k+5}), 
\end{eqnarray*}
which is a multiple of type (4) generator. The case of $j=q$ is similar. \\
{\bf Case 5:} 
Let $f=x_1x_{k+2}x_{2k+4}x_{2k+5}-\underline{x_{k+1}x_{2k+2}x_{2k+3}^2}$ 
be a generator of type (2), 
and $g$ a generator of type (3). First we consider the case where 
$g=x_px_{k+2}x_{2k+4}-\underline{x_{k+1}x_{k+1+p}x_{2k+3}}$. Then 
\begin{eqnarray*}
S(f,g)&=& -x_{k+1+p}f - (-x_{2k+2}x_{2k+3})g \\
&=& x_{k+2}x_{2k+4}(x_px_{2k+2}x_{2k+3}-x_1x_{k+1+p}x_{2k+5}), 
\end{eqnarray*}
where $x_px_{2k+2}x_{2k+3}-x_1x_{k+1+p}x_{2k+5}$ is of type (3). 
Next, let $g=\underline{x_px_{2k+2}x_{2k+3}}-x_1x_{k+1+p}x_{2k+5}$. Then 
\begin{eqnarray*}
S(f,g)&=& -x_pf - x_{k+1}x_{2k+3}g \\
&=& -x_1x_{2k+5}(x_px_{k+2}x_{2k+4}-x_{k+1}x_{k+1+p}x_{2k+3}) 
\end{eqnarray*}
and we have that $x_px_{k+2}x_{2k+4}-x_{k+1}x_{k+1+p}x_{2k+3}$ is of type (3). \\
{\bf Case 6:} Let $f$ be the same as in Case 5 and 
$g=\underline{x_px_qx_{k+2}x_{2k+2}x_{2k+4}}-x_1x_{k+1}x_{k+1+p}x_{k+1+q}x_{2k+5}$ 
be of type (4) generator, where $p \leq q$. Then 
\begin{eqnarray*}
S(f,g)&=& -x_px_qx_{k+2}x_{2k+4}f - x_{k+1}x_{2k+3}^2g \\
&=& -x_1x_{2k+5}(x_px_qx_{k+2}^2x_{2k+4}^2-
\underline{x_{k+1}^2x_{k+1+p}x_{k+1+q}x_{2k+3}^2}) \\
&=& -x_1x_{2k+5}\{x_{k+1}x_{k+1+q}x_{2k+3}(x_px_{k+2}x_{2k+4}-x_{k+1}x_{k+1+p}x_{2k+3}) \\
&\quad&\quad\quad +x_px_{k+2}x_{2k+4}(x_qx_{k+2}x_{2k+4}-x_{k+1}x_{k+1+q}x_{2k+3})\}. 
\end{eqnarray*}
Thus $S(f,g)$ reduce to 0 by generators of type (3). \\
{\bf Case 7:} We assume that both $f$ and $g$ are of type (3). 
First, we consider the case where 
$f=x_ix_{k+2}x_{2k+4}-\underline{x_{k+1}x_{k+1+i}x_{2k+3}}$ 
and $g=x_px_{k+2}x_{2k+4}-\underline{x_{k+1}x_{k+1+p}x_{2k+3}}$, 
where $i \neq p$. Then 
\begin{eqnarray*}
S(f,g)&=& -x_{k+1+p}f - (-x_{k+1+i})g \\
&=& -x_{k+2}x_{2k+4}(x_ix_{k+1+p}-x_px_{k+1+i}), 
\end{eqnarray*}
which is a multiple of type (1) generator. 
Next, let $f$ be the same one and 
$g=\underline{x_px_{2k+2}x_{2k+3}}-x_1x_{k+1+p}x_{2k+5}$. Then 
\begin{eqnarray*}
S(f,g)&=& -x_px_{2k+2}f - x_{k+1}x_{k+1+i}g \\
&=& -(x_ix_px_{k+2}x_{2k+2}x_{2k+4}-x_1x_{k+1}x_{k+1+i}x_{k+1+p}x_{2k+5}), 
\end{eqnarray*}
which is a generator of type (4) up to sign. 
Finally, let $f=\underline{x_ix_{2k+2}x_{2k+3}}-x_1x_{k+1+i}x_{2k+5}$ 
and $g=\underline{x_px_{2k+2}x_{2k+3}}-x_1x_{k+1+p}x_{2k+5}$, 
where $i \neq p$. Then 
\begin{eqnarray*}
S(f,g)&=&x_pf-x_ig \\
&=&x_1x_{2k+5}(x_ix_{k+1+p}-x_px_{k+1+i}). 
\end{eqnarray*}
{\bf Case 8:} Let $f$ be of type (3) 
and $g=\underline{x_px_qx_{k+2}x_{2k+2}x_{2k+4}}-x_1x_{k+1}x_{k+1+p}x_{k+1+q}x_{2k+5}$ 
be of type (4) with $p \leq q$. 
First, we set that $f=x_ix_{k+2}x_{2k+4}-\underline{x_{k+1}x_{k+1+i}x_{2k+3}}$. 
Then the leading monomials of $f$ and $g$ are coprime. 
Next, we set that $f=\underline{x_ix_{2k+2}x_{2k+3}}-x_1x_{k+1+i}x_{2k+5}$. 
If $i \not= p$ and $i \not= q$, then 
\begin{eqnarray*}
S(f,g)&=&x_px_qx_{k+2}x_{2k+4}f-x_ix_{2k+3}g \\
&=&x_1x_{2k+5}(\underline{x_ix_{k+1}x_{k+1+p}x_{k+1+q}x_{2k+3}}-x_px_qx_{k+2}x_{k+1+i}x_{2k+4}) \\
&=&x_1x_{2k+5}\{-x_ix_{k+1+q}(x_px_{k+2}x_{2k+4}-x_{k+1}x_{k+1+p}x_{2k+3}) \\
&\quad&\quad\quad +x_px_{k+2}x_{2k+4}(x_ix_{k+1+q}-x_qx_{k+1+i})\}. 
\end{eqnarray*}
Thus $S(f,g)$ reduce to 0 by generators of type (1) and (3). 
If $i=p$, then 
\begin{eqnarray*}
S(f,g)&=&x_qx_{k+2}x_{2k+4}f-x_{2k+3}g \\
&=&-x_1x_{k+1+i}x_{2k+5}(x_qx_{k+2}x_{2k+4}-x_{k+1}x_{k+1+q}x_{2k+3}). 
\end{eqnarray*}
The case of $i=q$ is similar. \\
{\bf Case 9:} Finally, we consider the case that both $f$ and $g$ are of type (4). 
Let $f=\underline{x_ix_jx_{k+2}x_{2k+2}x_{2k+4}}-x_1x_{k+1}x_{k+1+i}x_{k+1+j}x_{2k+5}$ and \\ 
$g=\underline{x_px_qx_{k+2}x_{2k+2}x_{2k+4}}-x_1x_{k+1}x_{k+1+p}x_{k+1+q}x_{2k+5}$, 
where $i \leq j$ and $p \leq q$. 
Without loss of generality, we may assume that $j \geq q.$ 
First, we assume that $j > q (\geq p)$. 
If $i \not= p$ and $i \not= q$, then
\begin{eqnarray*}
S(f,g)&=&x_px_qf-x_ix_jg \\
&=&x_1x_{k+1}x_{2k+5}(\underline{x_ix_jx_{k+1+p}x_{k+1+q}}-x_px_qx_{k+1+i}x_{k+1+j}) \\
&=&x_1x_{k+1}x_{2k+5}\{-x_ix_{k+1+q}(x_px_{k+1+j}-x_jx_{k+1+p})
+x_px_{k+1+j}(x_ix_{k+1+q}-x_qx_{k+1+i})\}. 
\end{eqnarray*}
Thus we have that $S(f,g)$ reduce to 0 by generators of type (1). 
If $i=p$, then 
\begin{eqnarray*}
S(f,g)&=&x_qf-x_jg \\
&=&x_1x_{k+1}x_{k+1+i}x_{2k+5}(x_jx_{k+1+q}-x_qx_{k+1+j}). 
\end{eqnarray*}
The case of $i=q$ is similar. 
Next, we consider the case where $j=q$. Then $i \neq p$ and 
\begin{eqnarray*}
S(f,g)&=&x_pf-x_ig \\
&=&x_1x_{k+1}x_{k+1+j}x_{2k+5}(x_ix_{k+1+p}-x_px_{k+1+i}), 
\end{eqnarray*}
which is a multiple of type (1) generator. 
\end{proof}

\begin{Corollary}
The initial ideal of $I_G$ with respect to $<_{\rev}$ 
is generated by the following monomials: 
\begin{eqnarray*}
&&x_jx_{k+1+i}, \quad 2 \leq i < j \leq k, \\ 
&&x_{k+1}x_{2k+2}x_{2k+3}^2, \\
&&x_{k+1}x_{k+1+r}x_{2k+3}, \  x_rx_{2k+2}x_{2k+3}, \quad 2 \leq r \leq k, \\
&&x_px_qx_{k+2}x_{2k+2}x_{2k+4}, \quad 2 \leq p \leq q \leq k. 
\end{eqnarray*}
\end{Corollary}
\begin{proof}
By Lemma \ref{Grobner}, the set of the binomials corresponding to 
primitive even closed walks (1), (2), (3) and (4) is a Gr\"{o}bner basis of $I_G$ 
with respect to $<_{\rev}$. 
Thus the leading terms with respect to $<_{\rev}$ 
generate the initial ideal of $I_G$. 
\end{proof}

For the rest part of this paper, 
we will denote by $I$, the initial ideal of $I_G$ with respect to $<_{\rev}$. 

\section{Proof of $\depth K[\xb]/I \leq 6$}

In this section, we will prove that $\depth K[\xb]/I \leq 6$. 
Since the number of edges of $G$, which coincides with $2k+5$, 
is equal to the number of variables of $K[\xb]$, 
Auslander--Buchsbaum formula implies that we may prove that $\pd K[\xb]/I \geq 2k-1$, 
where $\pd K[\xb]/I$ stands for the projective dimension of $K[\xb]/I$. 


First, we recall from \cite{CoCoA} 
the fundamental technique to compute the Betti numbers of 
(non-squarefree) monomial ideals. 

%
%
For a multi degree ${\bf a} \in \ZZ^n_{\geq 0}$, define 
$$
{\bf K}^{\bf a}(J)=\{\text{squarefree vectors } \alpha: {\bf x}^{{\bf a}-\alpha} \in J\}
$$
to be the \textit{Koszul simplicial complex} of $J$ in degree {\bf a}, 
where a squarefree vector $\a$ means that each entry of $\a$ is 0 or 1. 

\begin{Lemma}{\em (\cite[Theorem 1.34]{CoCoA})}\label{Koszul}
Let $S$ be a polynomial ring, $J$ a monomial ideal of $S$ 
and ${\bf a} \in \ZZ^n_{\geq 0}$ a vector. 
Then the Betti numbers of $J$ and $S/J$ in degree {\bf a} can be expressed as 
$$
\beta_{i,\ab}(J)=\beta_{i+1,\ab}(S/J)=\dim_K \Tilde{H}_{i-1}({\bf K}^{\ab}(J);K). 
$$
\end{Lemma}

%

By virtue of Lemma \ref{Koszul}, 
in order to prove that $\pd K[\xb]/I \geq 2k-1$, 
we may show the following 

\begin{Lemma}\label{main1}
Let $\ab=\sum_{j=2}^k(\eb_j+\eb_{k+1+j})+\eb_{k+1}+\eb_{2k+2}+2\eb_{2k+3} \in \ZZ^{2k+5}_{\geq 0}$, 
where $\eb_i \in \RR^{2k+5}$ is the $i$-th unit vector of $\RR^{2k+5}$. Then 
$$\dim_K \Tilde{H}_{2k-3}({\bf K}^{\ab}(I);K) \not= 0.$$
\end{Lemma}
\begin{proof}
Let $\Delta$ be the simplicial complex on the vertex set $[2k+5]$ which 
is obtained by identifying a squarefree vector $\alpha \in {\bf K}^{\ab} (I)$ 
with the set of coordinates where the entries of $\alpha$ are $1$. 
To prove the assertion, we may show that 
$\dim_K \tilde{H}_{2k-3} ({\Delta}; K) \neq 0$. 
Let $I_1$ (resp.\  $I_2$) be the monomial ideal generated by the monomials 
\begin{eqnarray*}
&&x_jx_{k+1+i}, \quad 2 \leq i < j \leq k, \\
&&x_{k+1}x_{k+1+r}x_{2k+3}, \  x_rx_{2k+2}x_{2k+3}, \quad 2 \leq r \leq k 
\end{eqnarray*}
(resp.\  by the monomial $x_{k+1}x_{2k+2}x_{2k+3}^2)$. 
We denote by $\Delta_1,\Delta_2$, the subcomplexes of $\Delta$ 
corresponding to ${\bf K}^{\ab} (I_1), {\bf K}^{\ab} (I_2)$, respectively. 
Since the $(k+2)$-th entry of $\ab$ is equal to 0, 
one has $\Delta=\Delta_1 \union \Delta_2$. 
Moreover, one can verify that all the facets of $\Delta_1$ 
contain a common vertex $2k+3$. 
In other words, $\Delta_1$ is a cone over some simplicial complex. 
In addition, $\Delta_2$ has only one facet 
$$\{2,3,\ldots,k,k+3,k+4,\ldots,2k+1\},$$ 
which is a $(2k-3)$-dimensional simplex. 
Thus the reduced homologies of both of $\Delta_1$ and $\Delta_2$ 
all vanish. Hence the Mayer--Vietoris sequence 
\begin{displaymath}
\begin{aligned}
\cdots 
&\longrightarrow \tilde{H}_{i} (\Delta_1 \cap \Delta_2; K)
\longrightarrow \tilde{H}_i (\Delta_1; K) \oplus \tilde{H}_i (\Delta_2; K)
\longrightarrow \tilde{H}_i (\Delta; K) \\
&\longrightarrow \tilde{H}_{i-1} (\Delta_1 \cap \Delta_2; K)
\longrightarrow 
\tilde{H}_{i-1} (\Delta_1; K) \oplus \tilde{H}_{i-1} (\Delta_2; K)
\longrightarrow \cdots
\end{aligned}
\end{displaymath}
yields 
\begin{displaymath}
\tilde{H}_i (\Delta; K) \cong \tilde{H}_{i-1} (\Delta_1 \cap \Delta_2; K)
\qquad \text{for all $i$.}
\end{displaymath}

Now we note that subsets 
\begin{eqnarray*}
&&\{2,3,\ldots,k,k+3,k+4,\ldots,2k+1\} \setminus \{i\}, \quad\quad i=2,\ldots,k, \\
&&\{2,3,\ldots,k,k+3,k+4,\ldots,2k+1\} \setminus \{k+1+j\}, \quad\quad j=2,\ldots,k 
\end{eqnarray*}
are faces of $\Delta_1$ and $\{2,3,\ldots,k,k+3,k+4,\ldots,2k+1\}$ 
is not a face of $\Delta_1$. Thus the above subsets are 
the facets of $\Delta_1 \cap \Delta_2$. 
In particular, one has $\dim (\Delta_1 \cap \Delta_2)=2k-4$. 
Since $\Delta_1 \cap \Delta_2$ contains all facets of 
the $(2k-3)$-dimensional simplex $\Delta_2$, 
the geometric realization of $\Delta_1 \cap \Delta_2$ is homeomorphic to 
the boundary complex of the simplex $\Delta_2$, i.e., 
$\Delta_1 \cap \Delta_2$ is a simplicial $(2k-4)$-sphere. 

Therefore, one has 
$\dim_K\tilde{H}_{2k-3} (\Delta; K) 
= \dim_K\tilde{H}_{2k-4} (\Delta_1 \cap \Delta_2; K) 
\not= 0.$ 
\end{proof}

\section{Proof of $\depth K[\xb]/I \geq 6$}


In this section, we will prove the following 
\begin{Lemma}\label{main2}
$$\depth K[\xb]/I \geq 6.$$
\end{Lemma}

Before proving Lemma \ref{main2}, we prepare the following two lemmas. 

Let $J \subset S = K[x_1,\ldots,x_n]$ be a monomial ideal. 
We denote by $G(J)$, the minimal set of monomial generators of $J$. 

\begin{Lemma}\label{syzygy}
Let $S=K[x_1,\ldots,x_n]$ be the polynomial ring in $n$ variables 
and $J \subset S$ a monomial ideal of $S$. 
\begin{enumerate}
\item[{\em (a)}] 
If only $m (\leq n)$ variables appear in the elements of $G(J)$, 
then $\depth S/J \geq n-m$. 
\item[{\em (b)}] 
If only $m$ variables appear in the elements of $G(J)$ 
and the variables $x_{i_1},\ldots,x_{i_r}$ 
do not appear 
in there, 
then $\depth S/J' \geq n-m$, where 
$J'=x_{i_1} \cdots x_{i_r}J$. 
\end{enumerate}
\end{Lemma}
\begin{proof}
Without loss of generality, we may assume that 
only the variables $x_1,\ldots,x_m$ appear in the elements of $G(J)$. \\
(a) Since the variables $x_{m+1},\ldots,x_n$ do not appear in the elements of 
$G(J)$, the sequence $x_{m+1},\ldots,x_n$ is an $S/J$-regular sequence. 
Thus one has $\depth S/J \geq n-m$. \\
(b) Set $x_{i_{\ell}}=x_{m+\ell}$ for $\ell=1,\ldots,r$ 
and $J''=(x_{m+1} \cdots x_{m+r}) \subset S$. 
Then, by the short exact sequence 
$0 \rightarrow S/J \cap J'' \rightarrow S/J \oplus S/J'' 
\rightarrow S/(J + J'') \rightarrow 0$, we have 
$$
\depth S/J' = \depth S/J \cap J'' \geq 
\min\{\depth S/J, \depth S/J'', \depth S/(J+J'')+1\}. 
$$
Now we have $\depth S/J \geq n-m$ by (a) and $\depth S/J'' = n-1$. In addition, 
since $x_{m+1},\ldots,x_{m+r}$ do not appear in the elements of $G(J)$, 
the monomial $x_{m+1} \cdots x_{m+r}$ is an $S/J$-regular element. 
Hence one has $\depth S/(J+J'') = \depth S/J-1 \geq n-m-1$. 
\end{proof}


Let 
\begin{eqnarray*}
&&I_1=(x_jx_{k+1+i}: 2 \leq i < j \leq k), \\
&&I_2=(x_{k+1}x_{2k+2}x_{2k+3}^2), \\
&&I_3=x_{2k+2}x_{2k+3}(x_2,x_3,\ldots,x_k), \\
&&I_4=x_{k+1}x_{2k+3}(x_{k+3},x_{k+4},\ldots,x_{2k+1}), \\
&&I_5=x_{k+2}x_{2k+2}x_{2k+4}(x_2,x_3,\ldots,x_k)^2. 
\end{eqnarray*}
Then $I=I_1+I_2+\cdots+I_5$. 

The following lemma can be obtained by elementary computations. 
\begin{Lemma}\label{ideals}
Let $J_1=I_3+I_4$, $J_2=J_1+I_1$ and $J_3=J_2+I_5$. Then \\
{\em (a)} $I_3 \cap I_4 = x_{k+1}x_{2k+2}x_{2k+3}(x_2,\ldots,x_k)(x_{k+3},\ldots,x_{2k+1})$. \\
{\em (b)} $J_1 \cap I_1 = x_{2k+3}(x_{k+1},x_{2k+2})I_1.$ \\
{\em (c)} $J_2 \cap I_5 = x_{k+2}x_{2k+2}x_{2k+4}(x_2,\ldots,x_k)(x_{2k+3}(x_2,\ldots,x_k)+I_1).$ \\
{\em (d)} $J_3 \cap I_2 = x_{k+1}x_{2k+2}x_{2k+3}^2(x_2,\ldots,x_k,x_{k+3},\ldots,x_{2k+1}).$ 
\end{Lemma}

Now we will prove Lemma \ref{main2}.
\begin{proof}[Proof of Lemma \ref{main2}]
Work with the same notations as in Lemma \ref{ideals}. 
By the short exact sequence 
$$0 \rightarrow K[\xb]/J_3 \cap I_2 \rightarrow K[\xb]/J_3 \oplus K[\xb]/I_2 
\rightarrow K[\xb]/(J_3 + I_2) \rightarrow 0,$$ one has 
\begin{displaymath}
  \begin{aligned}
    \depth K[\xb]/I 
    &= \depth K[\xb]/(J_3 + I_2) \\
    &\geq \min\{\depth K[\xb]/J_3,\depth K[\xb]/I_2, \depth K[\xb]/J_3 \cap I_2-1\}. 
  \end{aligned}
\end{displaymath}
Thus what we must prove is that the inequalities 
$\depth K[\xb]/J_3 \geq 6$, $\depth K[\xb]/I_2 \geq 6$ 
and $\depth K[\xb]/J_3 \cap I_2 \geq 7.$ 
Obviously, $\depth K[\xb]/I_2=2k+4 \geq 6$. 
Moreover, by Lemmas \ref{ideals} (d) and \ref{syzygy} (b), 
we can easily see that $\depth K[\xb]/J_3 \cap I_2 \geq (2k+5)-2(k-1) =7$. 
Thus we investigate $\depth K[\xb]/J_3$. 

{\bf (First step)} 
By the short exact sequence 
$$0 \rightarrow K[\xb]/I_3 \cap I_4 \rightarrow K[\xb]/I_3 \oplus K[\xb]/I_4 
\rightarrow K[\xb]/(I_3 + I_4) \rightarrow 0,$$ one has 
\begin{displaymath}
  \begin{aligned}
    \depth K[\xb]/J_1 
    &=\depth K[\xb]/(I_3 + I_4) \\
    &\geq \min\{\depth K[\xb]/I_3,\depth K[\xb]/I_4, \depth K[\xb]/I_3 \cap I_4-1\}.
  \end{aligned}
\end{displaymath}
By Lemma \ref{syzygy} (b), one has 
$\depth K[\xb]/I_3 \geq k+6 \geq 6$ and $\depth K[\xb]/I_4 \geq k+6 \geq 6$. 
Since $I_3 \cap I_4 
= x_{k+1}x_{2k+2}x_{2k+3}(x_2,\ldots,x_k)(x_{k+3},\ldots,x_{2k+1})$ 
by Lemma \ref{ideals} (a) and $x_{k+1},x_{2k+2},x_{2k+3}$ do not appear 
in the elements of $G((x_2,\ldots,x_k)(x_{k+3},\ldots,x_{2k+1}))$, 
one has $\depth K[\xb]/I_3 \cap I_4 \geq (2k+5) - 2(k-1) = 7$ by Lemma \ref{syzygy} (b). 
Hence one has $\depth K[\xb]/J_1 \geq 6$. 

{\bf (Second step)} 
Again, by the short exact sequence 
$$0 \rightarrow K[\xb]/J_1 \cap I_1 \rightarrow K[\xb]/J_1 \oplus K[\xb]/I_1 
\rightarrow K[\xb]/(J_1 + I_1) \rightarrow 0,$$ one has 
\begin{displaymath}
  \begin{aligned}
    \depth K[\xb]/J_2 &=\depth K[\xb]/(J_1 + I_1) \\
    &\geq \min\{\depth K[\xb]/J_1,\depth K[\xb]/I_1, \depth K[\xb]/J_1 \cap I_1-1\}.
  \end{aligned}
\end{displaymath}
By Lemma \ref{syzygy} (a), $\depth K[\xb]/I_1 \geq (2k+5)-2(k-2) \geq 6$. 
Also by Lemma \ref{ideals} (b), one has 
$J_1 \cap I_1 = x_{2k+3}(x_{k+1},x_{2k+2})I_1.$ 
Since only $2k-2$ variables appear in the elements of 
$G((x_{k+1},x_{2k+2})I_1)$, 
and $x_{2k+3}$ does not appear in there, 
one has 
$\depth K[\xb]/J_1 \cap I_1 \geq 7$ by Lemma \ref{syzygy} (b). 
In addition, one has $\depth K[\xb]/J_1 \geq 6$ by the first step. 
Hence one has $\depth K[\xb]/J_2 \geq 6$. 

{\bf (Third step)} 
Similarly, by the short exact sequences 
$$0 \rightarrow K[\xb]/J_2 \cap I_5 \rightarrow K[\xb]/J_2 \oplus K[\xb]/I_5 
\rightarrow K[\xb]/(J_2 + I_5) \rightarrow 0,$$ one has 
\begin{displaymath}
  \begin{aligned}
    \depth K[\xb]/J_3 &= \depth K[\xb]/(J_2 + I_5) \\
    &\geq \min\{\depth K[\xb]/J_2,\depth K[\xb]/I_5, \depth K[\xb]/J_2 \cap I_5-1\}.
  \end{aligned}
\end{displaymath}
By Lemma \ref{syzygy} (b), one has $\depth K[\xb]/I_5 \geq k+6 \geq 6$. 
For $\depth K[\xb]/J_2 \cap I_5$, by Lemma \ref{ideals} (c), one has 
$J_2 \cap I_5 = x_{k+2}x_{2k+2}x_{2k+4}(x_2,\ldots,x_k)(x_{2k+3}(x_2,\ldots,x_k)+I_1).$
Notice that only $2k-2$ variables appear and 
$x_{k+2},x_{2k+2},x_{2k+4}$ do not appear in the elements of 
$G((x_2,\ldots,x_k)(x_{2k+3}(x_2,\ldots,x_k)+I_1))$. 
Thus, again by Lemma \ref{syzygy} (b), 
one has $\depth K[\xb]/J_2 \cap I_5 \geq 7$. 
Combining these results with the second step, 
one has $\depth K[\xb]/J_3 \geq 6$. 

Therefore, one has $\depth K[\xb]/I \geq 6$, as required. 
\end{proof}

\section{Cohen--Macaulayness of $K[\xb]/\ini_{<_{\lex}}(I_G)$}

In this section, we will prove the following 
\begin{Lemma}\label{main3}
Let $<_{\lex}$ denote the lexicographic order on $K[\xb]$ induced by 
the ordering $x_1 > \cdots > x_{2k+5}$ of the variables. 
Then $K[\xb]/\ini_{<_{\lex}}(I_G)$ is Cohen--Macaulay. 
\end{Lemma}


First of all, we need to know the generators of $\ini_{<_{\lex}}(I_G)$. 
As an analogue of Lemma \ref{Grobner}, we can prove the following 
\begin{Lemma}
The set of binomials corresponding to primitive even closed walks 
(1), (2), (3) and (4) (appeared in section 1) 
is a Gr\"{o}bner basis of $I_G$ 
with respect to $<_{\lex}$. 
\end{Lemma}


\begin{Corollary}\label{initial_lex}
The initial ideal of $I_G$ with respect to $<_{\lex}$ 
is generated by the following monomials: 
\begin{displaymath}
  (\flat) \qquad
  \begin{aligned}
    &x_ix_{k+1+j}, \quad  2 \leq i < j \leq k, \\
    &x_1x_{k+2}x_{2k+4}x_{2k+5}, \\
    &x_rx_{k+2}x_{2k+4}, \  x_1x_{k+1+r}x_{2k+5}, \quad  2 \leq r \leq k. 
  \end{aligned}
\end{displaymath}

\par
In particular, $\ini_{<_{\lex}} I_G$ is a squarefree monomial ideal. 
\end{Corollary}

\par
Note that we can exclude the initial term of the binomial corresponding to 
the even closed walk of type (4). 

\bigskip

Let $I'$ be the initial ideal of $I_G$ with respect to $<_{\lex}$. 
Since $I'$ is squarefree, we can define a simplicial complex ${\Delta}'$ 
on $[2k+5]$ whose Stanley--Reisner ideal coincides with $I'$. 
In order to prove that $K[\xb]/I'$ is Cohen--Macaulay, 
we will show that $\Delta'$ is shellable. 

\par
We recall the definition of the shellable simplicial complex. 
Let $\Delta$ be a simplicial complex. 
We call $\Delta$ is \textit{pure} if every facets (maximal faces) 
of $\Delta$ have the same dimension. 
A pure simplicial complex $\Delta$ of dimension $d-1$ 
is called {\em shellable} if all its facets 
(those are all $(d-1)$-faces of $\Delta$) can be listed 
$$F_1,F_2,\ldots,F_s$$
in such a way that 
$$(\bigcup_{j=1}^{i-1} \langle F_j \rangle) \cap \langle F_i \rangle 
  \quad 
  \big( = \bigcup_{j=1}^{i-1} \langle F_j \cap F_i \rangle \big)$$
is pure of dimension $d-2$ for every $1 < i \leq s$. 
Here $\langle F_i \rangle:=\{\sigma \in \Delta : \sigma \subset F_i\}.$ 
It is known that if $\Delta$ is shellable, 
then $K[{\Delta}]$ is Cohen--Macaulay for any field $K$.

\par
To show that $\Delta'$ is shellable, we investigate the facets of $\Delta'$. 
Let $F({\Delta}')$ be the set of facets of $\Delta'$. 
Then the standard primary decomposition of $I' = I_{{\Delta}'}$ is 
$$
I_{{\Delta}'} = \bigcap_{F \in F(\Delta')} P_{\bar{F}}, 
$$
where $\bar{F}$ is the complement of $F$ in $[2k+5]$ and 
$P_{\bar{F}} = ( x_i : i \in \bar{F})$; see \cite[Lemma 1.5.4]{HerzogHibi}. 
Hence we can obtain $F({\Delta}')$ 
from the standard primary decomposition of $I'$. 

\begin{Lemma}\label{minimalprimes}
  The standard primary decomposition of $I'$ is the intersection of the 
  following prime ideals: 
  \begin{displaymath}
    (\sharp) \qquad 
    \begin{aligned}
      &(x_1)+(x_2,x_3,\ldots,x_k), \; 
       (x_{2k+5})+(x_2,x_3,\ldots,x_k), \\
      &(x_{k+2})+(x_{k+3},x_{k+4},\ldots,x_{2k+1}), \; 
       (x_{2k+4})+(x_{k+3},x_{k+4},\ldots,x_{2k+1}), \\
      &(x_1,x_{k+2})+I_{\ell}', \;\; 2 \leq \ell \leq k, \\
      &(x_1,x_{2k+4})+I_{\ell}', \;\; 2 \leq \ell \leq k, \\
      &(x_{k+2},x_{2k+5})+I_{\ell}', \;\; 2 \leq \ell \leq k, \\
      &(x_{2k+4},x_{2k+5})+I_{\ell}', \;\; 2 \leq \ell \leq k, 
    \end{aligned}
  \end{displaymath}
  where $I_{\ell}'=(x_2,\ldots,x_{\ell-1},x_{k+2+\ell},\ldots,x_{2k+1})$ 
  for $\ell=2,\ldots,k$. 
%
\end{Lemma}
\begin{proof}
  Since there is no relation of inclusion among the prime ideals 
  on $(\sharp)$, it is enough to prove that the intersection of 
  these prime ideals coincides with $I'$. 

  \par
  First, we consider the case where $k=1$. 
  Then $G(I') = \{ x_1 x_3 x_6 x_7 \}$ and $(\sharp)$ consist 
  of only the first $2$ rows: 
  $(x_1)$, $(x_{7})$, $(x_{3})$, and $(x_{6})$. 
  Thus the assertion trivially holds. 

  \par
  Next, we consider the case where $k=2$. Note that $I_{2}' = 0$. 
  Then the ideal $I'$ is 
  \begin{displaymath}
    \begin{aligned}
    I' &= (x_1 x_4 x_8 x_9, x_1 x_5 x_9, x_2 x_4 x_8) \\
       &= (x_1, x_2) \cap (x_1, x_4) \cap (x_1, x_8) 
         \cap (x_4, x_5) \cap (x_4, x_9) \cap (x_8, x_5) \cap (x_8, x_9) 
         \cap (x_9, x_2) \\
       &= (x_1, x_2) \cap (x_9, x_2) \cap (x_4, x_5) \cap (x_8, x_5) 
          \cap (x_1, x_4) \cap (x_1, x_8) \cap (x_4, x_9) \cap (x_8, x_9), 
    \end{aligned}
  \end{displaymath}
  as desired. 

  \par
  Hence we may assume that $k \geq 3$. 
  Then the intersection of the prime ideals on the first row of $(\sharp)$ is 
  \begin{displaymath}
    (x_1 x_{2k+5}, x_2, x_3, \ldots, x_k)
  \end{displaymath}
  and that on the second row of $(\sharp)$ is
  \begin{displaymath}
    (x_{k+2} x_{2k+4}, x_{k+3}, x_{k+4}, \ldots, x_{2k+1}). 
  \end{displaymath}
  For $\ell = 2, \ldots, k$, 
  the intersection of the prime ideals on the last $4$ rows of $(\sharp)$ is 
  \begin{displaymath}
    \begin{aligned}
      &((x_1, x_{k+2}) + I_{\ell}') \cap ((x_1, x_{2k+4}) + I_{\ell}') 
      \cap ((x_{k+2}, x_{2k+5}) + I_{\ell}') 
      \cap ((x_{2k+4}, x_{2k+5}) + I_{\ell}') \\
      = &((x_1, x_{k+2} x_{2k+4}) + I_{\ell}') 
         \cap ((x_{k+2} x_{2k+4}, x_{2k+5}) + I_{\ell}') \\
      = &(x_1 x_{2k+5}, x_{k+2} x_{2k+4}) + I_{\ell}'. 
    \end{aligned}
  \end{displaymath}
  Hence, the intersection of the prime ideals 
  on the last $4$ rows of $(\sharp)$ for all $\ell$ is 
  \begin{displaymath}
    (x_1 x_{2k+5}, x_{k+2} x_{2k+4}) + \bigcap_{\ell = 2}^k I_{\ell}'. 
  \end{displaymath}
  Therefore the intersection of all prime ideals of $(\sharp)$ is 
  \begin{equation}
    \label{eq:PrimeDecomp}
    \begin{aligned}
      &x_1 x_{2k+5} (x_{k+2} x_{2k+4}, x_{k+3}, x_{k+4}, \ldots, x_{2k+1}) 
        + x_{k+2} x_{2k+4} (x_{1} x_{2k+5}, x_2, x_3, \ldots, x_k) \\
      &+ (\bigcap_{\ell=2}^k I_{\ell}') \cap 
      (x_1 x_{2k+5}, x_2, x_3, \ldots, x_k)
      \cap (x_{k+2} x_{2k+4}, x_{k+3}, x_{k+4}, \ldots, x_{2k+1}). 
    \end{aligned}
  \end{equation}
  The ideal on the first row of (\ref{eq:PrimeDecomp}) coincides with 
  the one generated by monomials on the last $2$ rows 
  of $(\flat)$. 
  Since $I_2' = (x_{k+4}, x_{k+5}, \ldots, x_{2k+1})$ 
  and $I_k' = (x_2, x_3, \ldots, x_{k-1})$, 
  the ideal on the second row of (\ref{eq:PrimeDecomp}) coincides with 
  $\bigcap_{\ell = 2}^ {k} I_{\ell}'$. 
  Hence, we may prove that 
  \begin{displaymath}
    \bigcap_{\ell = 2}^k I_{\ell}' 
    = ( x_i x_{k+1+j} \; : \;  2 \leq i< j \leq k). 
  \end{displaymath}
  To show this equality, we prove 
  \begin{equation}
    \label{eq:intersectionI'}
    \bigcap_{\ell = 2}^{k'} I_{\ell}' 
    = (x_i x_{k+1+j} \; : \; 2 \leq i < j \leq k')
    + (x_{k+2+k'}, \ldots, x_{2k+1}) 
  \end{equation}
  for $k' = 2, \ldots, k$. When $k'=k$, we obtain the desired equality. 
  We use induction on $k' \geq 2$. 
  The case of $k'=2$ is trivial. 
  When (\ref{eq:intersectionI'}) holds for $k'$, we have 
  \begin{displaymath}
    \begin{aligned}
      \bigcap_{\ell = 2}^{k' + 1} I_{\ell}' 
      &= (\bigcap_{\ell = 2}^{k'} I_{\ell}') \cap I_{k' + 1}' \\
      &= ((x_i x_{k+1+j} \; : \; 2 \leq i < j \leq k') 
           + (x_{k+2+k'}, \ldots, x_{2k+1})) 
        \cap (x_2, \ldots, x_{k'}, x_{k+3+k'}, \ldots, x_{2k+1}) \\
      &= (x_i x_{k+1+j} \; : \; 2 \leq i < j \leq k')
          + x_{k+2+k'} (x_{2}, \ldots, x_{k'}) 
          + (x_{k+3+k'}, \ldots, x_{2k+1}) \\
      &= (x_i x_{k+1+j} \; : \; 2 \leq i < j \leq k'+1)
          + (x_{k+3+k'}, \ldots, x_{2k+1}), 
    \end{aligned}
  \end{displaymath}
  as desired. 
\end{proof}

\par
Now we are in the position to prove Lemma \ref{main3}. 
\begin{proof}[Proof of Lemma \ref{main3}]
By Lemma \ref{minimalprimes}, 
$F({\Delta}')$ consists of the following subsets of $[2k+5]$: 
\begin{eqnarray*}
&&F_1=\overline{\{1\} \union \{2,3,\ldots,k\}}, \; 
F_2=\overline{\{2k+5\} \union \{2,3,\ldots,k\}}, \\
&&F_3=\overline{\{k+2\} \union \{k+3,k+4,\ldots,2k+1\}}, \\ 
&&F_4=\overline{\{2k+4\} \union \{k+3,k+4,\ldots,2k+1\}}, \\ 
&&G_{1,\ell}=\overline{A_1 \union G_{\ell}'}, \;\; 2 \leq \ell \leq k, \\
&&G_{2,\ell}=\overline{A_2 \union G_{\ell}'}, \;\; 2 \leq \ell \leq k, \\
&&G_{3,\ell}=\overline{A_3 \union G_{\ell}'}, \;\; 2 \leq \ell \leq k, \\
&&G_{4,\ell}=\overline{A_4 \union G_{\ell}'}, \;\; 2 \leq \ell \leq k, 
\end{eqnarray*}
where $G_{\ell}'=\{2,\ldots,\ell-1,k+2+\ell,\ldots,2k+1\}$ for $2 \leq \ell \leq k$, 
$A_1=\{1,k+2\}$, $A_2=\{1,2k+4\}$, $A_3=\{k+2,2k+5\}$, $A_4=\{2k+4,2k+5\}$ 
and $\overline{F}=[2k+5] \setminus F$. 
Note that $G_{m, \ell} \cap A_j = \emptyset$ and $\#(G_{m,\ell}) = k-2$.  
In particular, $\Delta'$ is pure of dimension $k+4$. 

\par
Now we define the ordering on $F(\Delta')$ as follows:
\begin{eqnarray}\label{sequence}
G_{1,2},\ldots,G_{1,k},G_{2,2},\ldots,G_{2,k}, 
G_{3,2},\ldots,G_{3,k},G_{4,2},\ldots,G_{4,k}, 
F_1,F_2,F_3,F_4. 
\end{eqnarray}
We will prove $\Delta'$ satisfies the condition of shellability 
with this ordering. 
For $F, G \in F({\Delta})$, we write $G \prec F$ if $G$ lies in previous 
to $F$ on (\ref{sequence}). 

\par
First, we investigate 
$\Delta_{m, \ell} := (\bigcup_{G' \prec G_{m, \ell}} \langle G' \rangle) 
  \cap \langle G_{m,\ell} \rangle 
  = \bigcup_{G' \prec G_{m, \ell}} 
    \langle G' \cap G_{m,\ell} \rangle$ 
for $m = 1, 2, 3, 4$. For ${\ell}' < \ell$, one has 
\begin{eqnarray*}
G_{m,\ell'} \cap G_{m,\ell} 
&=& \overline{A_m \union G_{\ell'}'} \cap \overline{A_m \union G_{\ell}'} \\
&=& \overline{(A_m \union G_{\ell'}') \union (A_m \union G_{\ell}')} \\
&=& \overline{A_m \union \{2,\ldots,\ell-2,\ell-1,k+2+\ell',k+3+\ell',\ldots,2k+1\}} \\
&\subset& \overline{A_m \union \{2,\ldots,\ell-2,\ell-1,k+1+\ell,k+2+\ell,\ldots,2k+1\}} \\
&=& G_{m,\ell-1} \cap G_{m,\ell} 
\end{eqnarray*}
and $G_{m, \ell -1} \cap G_{m, \ell}$ 
is a $(k+3)$-dimensional face. 
Then we can conclude that $\Delta_{1, \ell}$ is pure of dimension $k+3$. 
Assume that $m = 2, 3, 4$. For $m' < m$, one has
\begin{eqnarray*}
G_{m',\ell'} \cap G_{m,\ell} 
&=& \overline{A_{m'} \union G_{\ell'}'} \cap \overline{A_{m} \union G_{\ell}'} \\
&=& \overline{(A_{m'} \union G_{\ell'}') \union (A_{m} \union G_{\ell}')} \\
&\subset& \overline{(A_{m'} \union A_m) \union G_{\ell}'}. 
\end{eqnarray*}
When $m=2$, then $m'=1$ and 
\begin{displaymath}
  \overline{(A_{1} \union A_2) \union G_{\ell}'}
  = \overline{\{ 1, k+2, 2k+4 \} \union G_{\ell}'} 
  = G_{1, \ell} \cap G_{2, \ell},
\end{displaymath}
which is $(k+3)$-dimensional. 
Therefore we can conclude that $\Delta_{2, \ell}$ is a pure simplicial complex 
of dimension $k+3$. 
Similarly, we can see that $\Delta_{m, \ell}$ is pure of dimension $k+3$ 
for $m=3, 4$ since e.g., 
$A_{2} \cup A_{3} \supset A_{1} \cup A_{3} = \{ 1, k+2, 2k+5 \}$. 

\par
Next, we investigate 
$\Delta_{s} := \bigcup_{G \prec F_s} \langle G \cap F_s \rangle$ 
for $s=1, 2, 3, 4$. 
It is easy to see that 
$G_{1,k} \cap F_1$ (resp.\  $G_{2,k} \cap F_1$) 
contains
$G_{1,\ell} \cap F_1$ and $G_{3,\ell} \cap F_1$ 
(resp.\  $G_{2,\ell} \cap F_1$ and $G_{4,\ell} \cap F_1$). 
Thus facets of $\Delta_1$ are $G_{1,k} \cap F_1$ and 
$G_{2,k} \cap F_1$, those are $(k+3)$-dimensional. 

\par
Similarly, we can see that the facets of $\Delta_2$ are
$G_{3,k} \cap F_1$, $G_{4,k} \cap F_1$,
and $F_1 \cap F_2$, those are also $(k+3)$-dimensional. 

\par
For $\Delta_3$, 
we can verify that 
$G_{1,2} \cap F_3$ (resp.\  $G_{3,2} \cap F_3$) 
is a $(k+3)$-dimensional face containing 
$G_{1,\ell} \cap F_3$, $G_{2,\ell} \cap F_3$ and $F_1 \cap F_3$ 
(resp.\  $G_{3,\ell} \cap F_3$, $G_{4,\ell} \cap F_3$ and $F_2 \cap F_3$). 
Therefore $\Delta_3$ is pure of dimension $k+3$. 

\par
Similarly, we can see that $\Delta_4$ is also 
a pure simplicial complex of dimension $k+3$ whose facets are  
$G_{2,2} \cap F_4$, $G_{4,2} \cap F_4$, 
and $F_3 \cap F_4$. 
\end{proof}

\end{document}